\documentclass[a4paper,11pt]{article}
\usepackage[T2A]{fontenc}
\usepackage[utf8]{inputenc}
\usepackage{amsmath,amssymb,amsthm}
\usepackage{amsfonts}
\usepackage{geometry}
\usepackage{url}
\usepackage[colorlinks, linkcolor=blue, citecolor=blue]{hyperref}

\theoremstyle{plain}
\newtheorem{theorem}{Theorem}[section]

\theoremstyle{definition}
\newtheorem{definition}{Definition}[section]

\theoremstyle{remark}
\newtheorem{remark}{Remark}[section]

\newtheorem{problem}{Problem}[section]

\newcommand{\Z}{\mathbb{Z}}

\title{Generalizations of the groups $G_{n}^{k}$: graphs, moduli spaces, algebraic geometry, spherical braids}
\author{V.~O.~Manturov\\
Moscow Institute of Physics and Technology (MIPT)}
\date{}

\begin{document}

\maketitle

\begin{center}
\textbf{UDC~515.162}
\end{center}

\begin{center}
\textbf{Keywords:} knot,braid group, 
configuration space, moduli space,
monodromy, spherical braid, pure braids, fundamental group, hypergraph, $G_n^k$-theory,elliptic curve.
\end{center}

\begin{center}
\textbf{Abstract}
\end{center}
In this work, we construct a generalization of the $G_n^k$-theory to the case of an arbitrary hypergraph. The case of spherical braids is considered separately, using the stratification of the moduli space $\mathcal{M}_n(S^2)$ and the hypergraph $\Gamma_n^{\mathrm{sph}}$ encoding projective constraints. In contrast to the original $G_{n}^{k}$ theory, where codimension-one properties are determined by exactly $k$ particles, the present work considers various cases corresponding to strata of codimension~$1$. 
These groups admit nice maps to free products of cyclic groups.

Among unsolved problems, we emphasize
the question how the above construction
works for abelian varieties and,
in particular, for elliptic curves.

\section{Introduction}

The main principle of the $G_n^k$-theory is as follows: 

{\em if a dynamical system describing the motion of $n$ particles admits a general-position property of codimension~$1$ depending on exactly $k$ particles, then this system has invariants taking values in the group $G_{n}^{k}$.}

This principle applies to various configuration spaces and moduli spaces in which the particles are ``equal'' and codimension~$1$ properties are determined by exactly $k$ of them. In the present work, this principle is generalized in several directions.

\begin{itemize}
    \item The most general case is that of an arbitrary hypergraph $\Gamma$, which makes it possible to cover situations where different sets of particles participate in codimension~$1$ conditions in different numbers.
    \item In the case of ``equal'' particles, graphs with a natural action of the symmetric group arise.
\end{itemize}

As special cases, we obtain the hypergraph $\Gamma_n^{\mathrm{sph}}$ and a homomorphism from the spherical pure braid group to the corresponding spherical analogue of the group $G_{n}^{3}$, as well as a new stratification of the ordinary configuration space of points in the plane, finer than that corresponding to $G_{n}^{3}$. This stratification yields a homomorphism from the braid group to the group $G_{n}^{6,3}$.

The approach to constructing invariants via stratification and monodromy has deep roots in the classical works of Arnold~\cite{arnold1969} and Birman~\cite{birman1975}, where the topology of configuration spaces and braid groups was first systematically studied. In particular, Arnold's computation of the cohomology ring of the colored braid group~\cite{arnold1969} and Birman's foundational monograph~\cite{birman1975} established the connection between braid groups and the topology of configuration spaces that underlies the present work.

Seemingly, this approach should lead
to lots of interesting graphs and
groups describing algebraic varieties
and various sorts of braids in them.

Elliptic curves are the first natural
case to be studied next.

\subsection{Acknowledgements}

I am very grateful to Igor Nikonov
and Matvey Sergeev,
and the referee for useful comments.

The author was supported by the grant 25-21-00884 Applied combinatorial geometry and topology.

In the compilation of this paper the author was partially using AI (YandexAlice and DeepSeek) (example and list of problems).

\section{The group $G(\Gamma)$ for an arbitrary hypergraph}

Let $\Gamma = (V, E)$ be a hypergraph, where $V$ is the set of vertices and $E$ is the set of hyperedges. We define the group $G(\Gamma)$ as the group given by generators and relations as follows.

For each vertex $v \in V$, we introduce a generator $g_v$. The set of generators is $\{ g_v \mid v \in V \}$, and we consider the relations:
\begin{enumerate}
    \item \textbf{Involutivity of generators.} For each vertex $v \in V$, we have
    \[
        g_v^2 = 1.
    \]

    \item \textbf{Hyperedge relations.} For each hyperedge $e \in E$ containing $k = |e|$ vertices, and for any permutation $(v_1, v_2, \dots, v_k)$ of the vertices of $e$, we have
    \[
        (g_{v_1} g_{v_2} \cdots g_{v_k})^2 = 1.
    \]
\end{enumerate}

Thus, the group $G(\Gamma)$ is given by the presentation
\begin{multline*}
    G(\Gamma) = \Bigl\langle \{ g_v \}_{v \in V} \;\Bigm|\;
    g_v^2 = 1\ (\forall v \in V), \\
    (g_{v_1} \cdots g_{v_k})^2 = 1\ (\forall e \in E,\ \forall (v_1,\dots,v_k) \text{ a permutation of the vertices of } e)
    \Bigr\rangle.
\end{multline*}

For instance, the group $G_n^k$ is obtained as $G(\Gamma_n^k)$ for the following hypergraph $\Gamma_n^k$. The vertex set is
$$
V(\Gamma_n^k)=\{\,v_I \mid I\subset \{1,\dots,n\},\ |I|=k\,\}.
$$
The hyperedges are of two types:
\begin{enumerate}
    \item for each pair $I,J\subset\{1,\dots,n\}$ with $|I|=|J|=k$ and $|I\cap J|<k-1$, we add an ordinary edge $\{v_I,v_J\}$;
    \item for each $(k+1)$-element subset $K\subset\{1,\dots,n\}$, we add a hyperedge
    $$
    e_K=\{\,v_I \mid I\subset K,\ |I|=k\,\}.
    $$
\end{enumerate}
Then $G(\Gamma_n^k)$ has generators $g_I$ for $I\subset\{1,\dots,n\}$, $|I|=k$, subject to $g_I^2=1$, the commutation relations $g_I g_J=g_J g_I$ whenever $|I\cap J|<k-1$, and the relations
$$
\Bigl(\prod_{I\subset K,\ |I|=k} g_I\Bigr)^2=1
$$
for every $(k+1)$-element subset $K\subset\{1,\dots,n\}$. This is precisely the presentation of $G_n^k$.

\subsection{On the mapping of the group $G(\Gamma)$ into the free product of groups~$\mathbb{Z}_2$}

One of the reasons why the
above groups are interesting is that that they admit easy
maps to free products of groups ${\mathbb Z}_{2}$, hence one can easily
extract powerful and
easy-to-calculate invariants
of braids and moduli spaces
hence such groups have exponential
growth and trivially solvable
word and conjugacy problems.

Let $G=(V,B)$ be a graph;
let $({\mathbb Z}_2)^{|V|}$ be the space
over ${\mathbb Z}_{2}$ generated
by all vertices $V$.
Let a word~$w$ composed of generators~$g_v$ of the group $G$ be given. Fix a vertex~$v$ of the graph. To each occurrence of the letter~$g_v$ in the word~$w$ we assign some index~$i$, which is an element of the vector space~$V(g_v, G(\Gamma))$ over the field~$\mathbb{Z}_2$.

To the word~$w$ we assign the product of elements of the form $g_v^i$, taken over all indices of occurrences of the letter~$g_v$; denote this product by~$A(w)$. In other words, when constructing~$A(w)$ we ignore all generators different from~$g_v$, and to each occurrence of~$g_v$ we attach the corresponding index.

Our goal is to construct the linear spaces~$V(g_v, G(\Gamma))$ in such a way that the element~$A(w)$ is well-defined in the free product of groups~$\mathbb{Z}_2$ indexed by elements of the space~$V(g_v, G(\Gamma))$. This means that for any words $w, w' \in G(\Gamma)$ representing the same group element, the corresponding elements $A(w)$ and $A(w')$ must coincide in the said free product.

To achieve this goal, we require that for each relation $R_{j}=1$ corresponding to a hyperedge of the graph~$\Gamma$, both occurrences of the generator~$g_v$ (within that relation) have the same index.

We describe the general construction of the spaces~$V(g_v, G(\Gamma))$. Consider the space of dimension~$|V|$ over~$\mathbb{Z}_2$ (which can be identified with~$\mathbb{Z}_2^{|V|}$), where each basis vector is assigned to a vertex of the graph $G(\Gamma)$. For each vertex~$v$, the space~$V(g_v, G(\Gamma))$ is defined as the quotient of the said space ${\mathbb {Z}}_{2}^{|V|}$ by the vector corresponding to $v$ itself and the subspace generated by the sums of vectors corresponding to the vertices in a hyperedge containing the vertex v
(relations $R_{j}$) containing the vertex~$v$. If $i$ is the corresponding element in the $V(g_{v},G(\Gamma))$,then we take $i$ to be the index for this occurency and write $g_{v}^{i}$.

Let $G(v, G(\Gamma))$ be
the group generated by
letters $g_{v}^{j}$ where
$j$ runs over $V(g_{v},G(\Gamma))$
modulo relations $(g_{v}^{j})^{2}=1$
(this is a free product of 
groups $\mathbb{Z}_{2}$).

\begin{theorem}
For each $v\in V(G(\Gamma))$ the map $Ind:
G(\Gamma)\to  G(v, G(\Gamma))$ taking $w$ to the product of
$g_{v}^{j}$, is well defined.
\end{theorem}

As an illustration, consider the example of the group~$G_n^3$. In this case we have generators $a_{ijk}$, and each generator $a_{ijk}$ commutes with all generators $a_{pqr}$ having less than two common indices with it. 
Such generators are factored out
one-by-one in the space
$V(g_{a_ijk},G)$.
The remaining generators come from the relations of the form $a_{ijk} a_{pqr} a_{ijk} a_{pqr} = 1$.

Let us turn to the construction of the quotient space. The original space $V$ is factored by the one-dimensional subspaces corresponding to index triples $pqr$ that satisfy the above commutation condition with a fixed triple $ijk$, and also by the subspace corresponding to the triple $ijk$ itself.

After such a factorisation, only generators corresponding to triples of the form $ijl$, $ikl$, $jkl$, where $l \neq i,j,k$, remain in the quotient space. Factoring by these directions leads to a decomposition of the space into a direct sum of two-dimensional subspaces, one for each index $l$ different from $i$, $j$, $k$.

The resulting structure coincides exactly with the set of MN-indices described in the work \cite{manturovnikonov2015}.

\section{Main theorem on homomorphisms from fundamental groups of moduli spaces to groups $G(\Gamma)$}

Let $\mathcal{M}$ some topological space (say, be a moduli space) equipped with a stratification such that:
\begin{itemize}
    \item codimension $1$ strata correspond to vertices of the hypergraph $\Gamma$;
    \item codimension $2$ strata correspond to hyperedges of $\Gamma$;
    \item in a neighbourhood of each codimension $2$ stratum corresponding to a hyperedge $e=\{v_1,\dots,v_m\}$, passes
    transverselly through this stratum of codimension 1 (hence it occurs exactly twice and from opposite sides).
\end{itemize}

We define a map $\Phi$ from the set of loops in $\mathcal{M}$ based at a point $*$ in the top-dimensional stratum to the group $G(\Gamma)$ as follows. For a loop $\gamma$ in general position (transverse to all strata, intersecting only finitely many codimension $1$ strata, and avoiding tangencies with strata of codimension $\ge 3$), we associate a word $w(\gamma)$ in the generators $g_v$ of $G(\Gamma)$ by reading off the intersections of $\gamma$ with codimension $1$ strata in the order they occur along the loop. Each intersection with a stratum corresponding to a vertex $v$ contributes the letter $g_v$; since $g_v^2=1$, we have $g_v=g_v^{-1}$, so the orientation of the intersection does not affect the letter. This defines a word $w(\gamma)$.

\begin{theorem}\label{thm:main}
The map $\Phi$ is well-defined on homotopy classes of loops and induces a group homomorphism
$$
\Phi_* \colon \pi_1(\mathcal{M}, *) \to G(\Gamma).
$$
\end{theorem}

\begin{proof}
We must show that if two loops $\gamma_0$ and $\gamma_1$ are homotopic, then $w(\gamma_0)$ and $w(\gamma_1)$ represent the same element in $G(\Gamma)$. It suffices to consider elementary homotopies:
\begin{enumerate}
    \item A homotopy that creates or cancels a tangency with a codimension $1$ stratum. Locally, this corresponds to a pair of transverse intersections with the same stratum in opposite directions, yielding the subword $g_v g_v$, which is trivial since $g_v^2=1$.
    \item A homotopy that moves a loop across a codimension $2$ stratum corresponding to a hyperedge $e=\{v_1,\dots,v_m\}$. By the condition that each vertex occurs exactly twice and from opposite sides, the local monodromy around such a stratum yields a relation
    $$
    (g_{v_1}\cdots g_{v_m})^2=1,
    $$
    which is exactly a defining relation of $G(\Gamma)$ for the hyperedge $e$.
\end{enumerate}
Since any homotopy can be decomposed into these elementary moves (by standard transversality and stratified Morse theory arguments), the word $w(\gamma)$ changes by insertions and deletions of relators of $G(\Gamma)$. Hence the image in $G(\Gamma)$ depends only on the homotopy class of $\gamma$. Therefore $\Phi$ descends to a map $\Phi_* \colon \pi_1(\mathcal{M},*) \to G(\Gamma)$. It is a homomorphism because concatenation of loops corresponds to concatenation of words (up to homotopy), and the group operation in $G(\Gamma)$ is multiplication of words.
\end{proof}

\begin{remark}
Unlike the trivial homomorphism, the map $\Phi_*$ is constructed from the geometry of the stratification and is generally nontrivial. For instance, in the case of $G_n^3$ and the configuration space of $n$ points in the plane, it recovers the known invariants.
\end{remark}

\section{Groups $G(\Gamma)$ for graphs $\Gamma$ with transitive action of the symmetric group $S_{n}$}

In this section, we consider hypergraphs $\Gamma$ whose vertices correspond to certain subsets of $\{1,\dots,n\}$, and whose hyperedges correspond to certain sets of vertices, with the symmetric group $S_n$ acting naturally on vertices and hyperedges. These hypergraphs correspond to configuration spaces of $n$ ``equal'' particles, and the action of the symmetric group is induced by permutations of elements of the base set $\{1,\dots,n\}$ and is assumed to be transitive on vertices and on each type of hyperedge separately.

The simplest examples of this theory are the standard $G_{n}^{3}$ and $G_{n}^{4}$ theories for ordinary braids in the plane, corresponding to codimension~$1$ events: ``three points lie on a line'' and ``four points lie on a circle/line'' A number of other examples are given in the book~\cite{manturov2020}, for instance, when one considers configurations of $n$ points in Euclidean (or projective) $(k-1)$-dimensional space such that any $k-1$ points are in general position. Considering the property ``$k$ points lie on a $(k-2)$-dimensional plane'' leads to a homomorphism from the corresponding group to the group $G_{n}^{k}$.

It is worth noting that the groups $G_n^2$ studied in \cite{manturov2017} are closely related to Coxeter groups. The standard reference for the theory of Coxeter groups is Humphreys' book \cite{humphreys1990}, which provides a comprehensive treatment of reflection groups and their Coxeter presentations. The connection between $G_n^2$ and Coxeter groups established by Manturov \cite{manturov2017} reveals an important algebraic structure underlying the more general $G(\Gamma)$ construction.

Formally, let $V$ be the vertex set of the hypergraph $\Gamma$, and suppose $S_n$ acts transitively on $V$ and on the set of edges respectively. Similarly, let the set of hyperedges $E$ be partitioned into types $E = E_1 \cup \dots \cup E_r$, with the action of $S_n$ also transitive on each $E_i$. Then the group $G(\Gamma)$ constructed from $\Gamma$ according to the definition in the previous section inherits this action: automorphisms of the hypergraph induced by $S_n$ determine automorphisms of the group $G(\Gamma)$.

A particular example of such a construction is the group $G_{n}^{k}$, which will be discussed further. In this construction, vertices correspond to $k$-element subsets of $\{1,\dots,n\}$, and the symmetric group $S_n$ acts transitively on them, while the types of hyperedges (pair edges and $(k+1)$-cliques) are also orbits of the $S_n$-action.

\section{Spherical braids and stratification of the moduli space}

Consider the moduli space of configurations of $n$ distinct points on the sphere:
\[
    \mathcal{M}_n(S^2) = \mathrm{Conf}_n(S^2)/\]
where $\mathrm{Conf}_n(S^2)$ is the space of ordered $n$-tuples of distinct points on $S^2$, and the action of $\mathrm{PGL}(2,\mathbb{C})$ corresponds to fractional-linear transformations (automorphisms of the sphere).

Identify $S^2 \setminus \{\infty\}$ with $\mathbb{R}^2$ via stereographic projection. Then a typical configuration is $n$ points in $\mathbb{R}^2$, no three of which are collinear.

We introduce a stratification of $\mathcal{M}_n(S^2)$:
\begin{itemize}
    \item \textbf{Main stratum (codimension~$0$):} configurations in which all $n$ points lie in $\mathbb{R}^2$ and no three are collinear.
    \item \textbf{Codimension~$1$ stratum:} configurations in which exactly one triple of points $(i,j,k)$ is collinear, while the remaining points are in general position.
    \item \textbf{Higher codimension strata:}
        \begin{enumerate}
            \item One point ``goes to infinity'' (does not lie in $\mathbb{R}^2$).
            \item There is a quadruple of points lying on a line in ${\mathbb R}^{2}$.
            \item There are two disjoint or intersecting triples of collinear points.
        \end{enumerate}
\end{itemize}

\subsection{Stratification hypergraph and the group $G(\Gamma_n^{\mathrm{sph}})$}

We construct the hypergraph $\Gamma_n^{\mathrm{sph}}$ encoding this stratification:
\begin{itemize}
    \item \textbf{Vertices:} triples of indices $\{i,j,k\}$, $1 \le i<j<k \le n$, corresponding to possible triples of collinear points (analogous to vertices for $G_n^3$).
    \item \textbf{Hyperedges:}
        \begin{enumerate}
            \item Ordinary hyperedges for the group $G_n^3$: $(k+1)$-cliques and pair edges according to intersection conditions.
            \item \textit{New hyperedges:} for each fixed index $i$, consider all triples containing $i$. This hyperedge contains all vertices of the form $\{i,j,k\}$ for $j,k \ne i$. Their number is $\binom{n-1}{2} = \frac{(n-1)(n-2)}{2}$.
        \end{enumerate}
\end{itemize}
Thus, the hypergraph contains both ``local'' relations (as in $G_n^3$) and ``global'' constraints related to the projective geometry of $S^2$.

\begin{definition}[Group $G(\Gamma_n^{\mathrm{sph}})$]\label{def:G_sph}
The group $G(\Gamma_n^{\mathrm{sph}})$ is defined by generators $a_{ijk}$ corresponding to vertices $\{i,j,k\}$, and relations:
\begin{enumerate}
    \item $a_{ijk}^2 = 1$ for all triples $\{i,j,k\}$.
    \item For any hyperedge and any permutation of its vertices $(v_1,\dots,v_m)$, we have $(a_{v_1}\cdots a_{v_m})^2 = 1$.
    \item In particular, for the ``new'' hyperedges corresponding to a fixed $i$, we have the relation:
    \[
        \left(\prod_{\substack{j<k\\ j,k\ne i}} a_{ijk}\right)^2 = 1,
    \]
    where the product is taken over all pairs $j,k$ different from $i$.
\end{enumerate}
\end{definition}

These relations reflect the fact that when a point $i$ is fixed, all possible collinearities involving $i$ are interconnected through the projective geometry of the sphere.
\subsection{Monodromy theorem for spherical braids}

Let $\mathcal{M}_n(S^2)$ be the moduli space of configurations of $n$ distinct points on the sphere, equipped with the stratification introduced in Section~5.1. Recall that codimension $1$ strata correspond to triples of collinear points, and codimension $2$ strata correspond to hyperedges of the hypergraph $\Gamma_n^{\mathrm{sph}}$.

Fix a basepoint $*$ in the main stratum. We define a map $\Phi$ from the set of loops in $\mathcal{M}_n(S^2)$ based at $*$ to the group $G(\Gamma_n^{\mathrm{sph}})$ as follows. Let $\gamma\colon [0,1]\to \mathcal{M}_n(S^2)$ be a loop in general position: it is transverse to all strata, intersects only finitely many codimension $1$ strata, and has no tangencies with strata of codimension $\ge 3$. Reading off the intersections of $\gamma$ with codimension $1$ strata in the order they occur along the loop, we associate to each intersection with the stratum corresponding to a triple $\{i,j,k\}$ the generator $a_{ijk}$. Since $a_{ijk}^2=1$, we have $a_{ijk}=a_{ijk}^{-1}$, so the orientation of the intersection does not affect the letter. Thus we obtain a word $w(\gamma)$ in the generators of $G(\Gamma_n^{\mathrm{sph}})$.

\begin{theorem}\label{thm:sph_monodromy}
The map $\Phi$ is well-defined on homotopy classes of loops and induces a group homomorphism
$$
\Phi_*\colon \pi_1(\mathcal{M}_n(S^2), *) \to G(\Gamma_n^{\mathrm{sph}}).
$$
\end{theorem}

\begin{proof}
The proof is a direct application of Theorem~\ref{thm:main}. Indeed, the moduli space $\mathcal{M}_n(S^2)$ with the above stratification satisfies all hypotheses of Theorem~\ref{thm:main}:
\begin{itemize}
    \item codimension $1$ strata correspond to vertices of $\Gamma_n^{\mathrm{sph}}$, i.e.\ to triples $\{i,j,k\}$;
    \item codimension $2$ strata correspond to hyperedges of $\Gamma_n^{\mathrm{sph}}$;
    \item in a neighbourhood of each codimension $2$ stratum, each vertex occurs exactly twice and from opposite sides. This follows from the local geometry of the stratification: when two collinearity conditions merge, each of the involved triples appears twice, and the monodromy around the codimension $2$ stratum yields exactly the relation $(a_{v_1}\cdots a_{v_m})^2=1$ for the corresponding hyperedge.
\end{itemize}
Therefore, by Theorem~\ref{thm:main}, there exists a natural homomorphism from $\pi_1(\mathcal{M}_n(S^2), *)$ to $G(\Gamma_n^{\mathrm{sph}})$. The map $\Phi$ defined above is precisely the map constructed in the proof of Theorem~\ref{thm:main}. Hence $\Phi$ is well-defined and induces the required homomorphism $\Phi_*$.
\end{proof}

\subsection{Comparison with the group $G_n^3$}

The group $G(\Gamma_n^{\mathrm{sph}})$ is an extension of $G_n^3$ that takes into account the specifics of spherical geometry. A comparison table is given below.

\begin{table}[h]
    \centering
    \begin{tabular}{lll}
        \hline
        \textbf{Object} & \textbf{$G_n^3$} & \textbf{$G(\Gamma_n^{\mathrm{sph}})$} \\ \hline
        Vertices & Triples $\{i,j,k\}$ & Triples $\{i,j,k\}$ \\
        Local hyperedges & $(k+1)$-cliques, pair edges & Same + additional \\
        Global relations & Absent & Hyperedge for each $i$: all triples with $i$ \\
        Geometric meaning & Planar configurations & Configurations on $S^2$ with projective constraints \\ \hline
    \end{tabular}
\end{table}

Thus, $G(\Gamma_n^{\mathrm{sph}})$ takes into account the projective symmetries of the sphere and allows one to distinguish spherical braids that would look identical in the plane.

\section{Why these groups are interesting}

Groups of the form $G(\Gamma)$ and their special cases $G_n^k$ are interesting from several points of view: topological, combinatorial, and algorithmic.

First, algebraically, as we have seen above,the groups $G_{n}^{k}$ and their generalisations defined
above admit
 {\em homomorphic maps to
free products of groups
$\Z_{2}$}.
These are constructed using the so-called
MN-indices (see
\cite{manturovnikonov2015}).

The groups $\Z_{2}*\cdots *\Z_{2}$
are convenient because they are sufficiently
large (have exponential growth)
and have easily solvable word
problems and conjugacy problems.

Thus, in all the cases above, we obtain powerful invariants
that are easy to compare.

From a topological point of view, the groups $G(\Gamma)$ arise as natural images of fundamental groups of stratified moduli spaces under a map that takes into account the local combinatorics of codimension~$1$ and~$2$ strata. Such a map encodes the monodromy of the configuration system: each intersection of a loop with a codimension~$1$ stratum corresponds to passing through a degenerate configuration (e.g., collinearity of three points or concyclicity of four), while loops around codimension~$2$ strata reflect relations arising from simultaneous degeneracy of several conditions. Thus, $G(\Gamma)$ serves as a discrete ``skeleton'' of the topology of the moduli space, preserving key homotopy invariants.

An important advantage is the ability to construct explicit homomorphisms from fundamental groups of configuration spaces to $G(\Gamma)$, yielding computable invariants for braids and their generalizations. For example, the map from the spherical braid group to $G(\Gamma_n^{\mathrm{sph}})$ allows one to distinguish elements of the braid group by how they intersect strata of collinearity and projective constraints.

Moreover, the groups $G(\Gamma)$ admit an action of the symmetric group $S_n$ when the hypergraph $\Gamma$ is constructed from equal particles. This action is compatible with the natural symmetry of the configuration space and allows the use of representation theory and invariant methods in studying the structure of the group.

It looks plausible that this approach
will work for study of other moduli
spaces (say, Grassman spaces).

\subsection{A nontrivial example: a loop around a fixed point}

\label{sec:example}

Consider the case \(n=4\). Fix three points \(1,2,3\) in general position (not lying on one line) and move the fourth point \(4\) along a small circle centered at the point \(1\). Choose the circle and the starting point so that, during one full turn, the point \(4\) crosses the collinearity strata in the order
\[
124,\quad 134,\quad 124,\quad 134,
\]
and does not cross the stratum corresponding to the triple \(234\). Such a loop exists: it is enough to take a sufficiently small circle around \(1\), which intersects the lines \(12\) and \(13\) twice each, but avoids the line \(23\).

This closed path in \(\mathcal{M}_4(S^2)\) corresponds to the word
\[
w=a_{124}a_{134}a_{124}a_{134}
=(a_{124}a_{134})^2
\]
in the group \(G(\Gamma_4^{\mathrm{sph}})\).

This word is nontrivial. Indeed, define a homomorphism
\[
\varphi\colon G(\Gamma_4^{\mathrm{sph}})\to D_8
=\langle x,y\mid x^2=y^2=(xy)^4=1\rangle
\]
(the dihedral group of order \(8\)) on the generators by
\[
\varphi(a_{123})=\varphi(a_{124})=x,\qquad
\varphi(a_{134})=\varphi(a_{234})=y.
\]
Let us check that all defining relations of \(G(\Gamma_4^{\mathrm{sph}})\) hold. Any product of three generators contains either two letters \(x\) and one letter \(y\), or one letter \(x\) and two letters \(y\); such a product is an involution. Any product of four generators contains two letters \(x\) and two letters \(y\), hence is equal either to \(1\) or to \((xy)^2\), which is also an involution. Therefore all relations are satisfied, and \(\varphi\) is well defined.

At the same time,
\[
\varphi(w)=\varphi\bigl((a_{124}a_{134})^2\bigr)
=(xy)^2\neq 1
\]
in \(D_8\). Hence \(w\neq 1\) in \(G(\Gamma_4^{\mathrm{sph}})\). Thus the constructed spherical braid gives a nontrivial element, and the monodromy homomorphism \(\Phi_*\) indeed carries nontrivial information.

\section{Open problems}

Within the framework of the theory of groups $G(\Gamma)$, a number of important open questions remain, concerning topology, algebra, and computability.

\begin{problem}
The condition formulated above
that codimension 1 strata are attached to codimension 2 strata
{\em exactly twice and from opposite sides}
is very natural and general.

As we see, it generalises the $G_{n}^{k}$-condition and holds in many other cases (say, for spherical braids).

It would be extremely interesting to understand to which extent this condition is typical in stratification which happens in algebraic geometry.

The most intriguing seems to be
when we consider an abelian manifold.
say, an elliptic curve and a set of
codimension $1$ relations in terms of the group operations (say, $2a+b=c$).

Will this satisfy the condition above?
If yes, what will the graph and the group look like?

\end{problem}

\begin{problem}
Study the structure of the kernel of the map $\Phi\colon \pi_1(\mathcal{M}, *) \to G(\Gamma)$. In particular, determine under which conditions on the stratification and the hypergraph $\Gamma$ this map is surjective, and under which it has a nontrivial kernel.
\end{problem}

\begin{problem}
Investigate the algorithmic decidability of the word problem in the groups $G(\Gamma)$ for various classes of hypergraphs.
\end{problem}

\begin{problem}
Understand for which spaces
the fundamental groups coincide
with the groups defined above
(rather than merely mapping into them).
\end{problem}

\begin{problem}
Construct and study fine invariants of elements of the group $G(\Gamma)$ that take into account not only intersections with codimension~$1$ strata, but also higher strata (codimension $\ge 3$). In particular, consider the possibility of introducing a filtration by codimension and the corresponding graded objects.
\end{problem}

\begin{problem}
For the spherical case $G(\Gamma_n^{\mathrm{sph}})$, determine whether there exists a homomorphism from the spherical braid group to $G(\Gamma_n^{\mathrm{sph}})$ that is an embedding (or at least has trivial kernel) for sufficiently large $n$.
\end{problem}

\begin{problem}
As we see, the groups provide many
invariants of {\em braids.} How can we obtain
from them invariants of {\em knots}?
\end{problem}

\begin{problem}
How can one study objects from
algebraic geometry using
the groups described above,
if the strata in the moduli
space are determined by an algebraic structure (e.g., addition of points on an elliptic curve)?
\end{problem}

\begin{problem}
Determine how properties of the hypergraph $\Gamma$ (e.g., its symmetries, connectivity, presence of cycles) are reflected in properties of the group $G(\Gamma)$: abelianization, centre, torsion, growth, etc.
\end{problem}

\begin{problem}
Understand how one can use the study of
spherical braids for understanding
non-stable homotopy groups of
spheres ($\pi_{n}(S^{2})$).
\end{problem}

Solving these problems will not only deepen our understanding of the structure of the groups $G(\Gamma)$, but also expand their scope of application in low-dimensional topology, braid theory, and algebraic geometry.


\begin{thebibliography}{9}

\bibitem{manturov2020} V.O. Manturov,
D.A. Fedoseev, S. Kim, I.M. Nikonov:
{\em Invariants and Pictures:
Low-dimensional topology and
combinatorial group theory},
World Scientific, 2020.

\bibitem{manturov2015}
V.O. Manturov, ``Non-Reidemeister knot theory and its applications in dynamical systems, geometry, and topology'', arXiv:
1501.05208.

\bibitem{manturovnikonov2015}
V.O. Manturov, I.M. Nikonov,
``On Braids and Groups 
$G_{n}^{k}$'', Journal of Knot Theory and Its Ramifications (Vol.~24, No.~13, 2015).

\bibitem{manturov2017}
V.O. Manturov,
``The groups $G_n^2$ and Coxeter groups'',
Russian Math. Surveys, 72:2 (2017), 378--380.

\bibitem{arnold1969}
V.I. Arnold,
``The cohomology ring of the colored braid group'',
Mat. Zametki (Math. Notes), 5:2 (1969), 227--231.

\bibitem{birman1975}
J.S. Birman,
{\em Braids, Links, and Mapping Class Groups},
Princeton University Press, 1975.
(Annals of Mathematics Studies, vol.~82).

\bibitem{humphreys1990}
J.E. Humphreys,
{\em Reflection Groups and Coxeter Groups},
Cambridge University Press, 1990.
(Cambridge Studies in Advanced Mathematics, vol.~29).

\end{thebibliography}
\end{document}